\documentclass[12pt,a4paper]{amsart}
\usepackage{amsfonts}
\usepackage[top=35mm, bottom=35mm, left=30mm, right=30mm]{geometry}
\usepackage[colorlinks=true,citecolor=blue]{hyperref}
\usepackage{mathptmx}
\usepackage{eucal}
\usepackage{graphicx}
\usepackage{amssymb}
\usepackage{amsmath}
\usepackage{amsthm}
\usepackage{enumerate}
\usepackage{verbatim}

\usepackage{xcolor}
\newtheorem{thm}{Theorem}[section]
\newtheorem{cor}[thm]{Corollary}
\newtheorem{lem}[thm]{Lemma}
\newtheorem{prop}[thm]{Proposition}

\newtheorem{ques}[thm]{Question}

\theoremstyle{definition}
\newtheorem{defn}[thm]{Definition}

\numberwithin{equation}{section}

\newcommand{\N}{\mathbb{N}}

\def \d {\delta}

\newcommand{\Z}{\mathbb{Z}}

\newcommand{\Orb}{\mathrm{Orb}}

\DeclareMathOperator{\supp}{supp}
\DeclareMathOperator{\diam}{diam}

\newcommand{\ep}{\epsilon}

\newcommand{\ra}{\rightarrow}

\def \B {\mathcal B}

 \makeatletter
\@namedef{subjclassname@2020}{%
	\textup{2020} Mathematics Subject Classification}
\makeatother

\begin{document}

\title{Sequence entropy tuples and mean sensitive tuples}
\author[J. Li]{Jie Li }
\address[Jie Li]{School of Mathematics and Statistics, Jiangsu Normal
University, Xuzhou, Jiangsu, 221116, P.R. China}
\email{jiel0516@mail.ustc.edu.cn}

\author[C. Liu]{Chunlin Liu }
\address[Chunlin Liu]{CAS Wu Wen-Tsun Key Laboratory of Mathematics, School of Mathematical Sciences, University of Science and Technology of China, Hefei, Anhui, 230026, P.R. China}
\email{lcl666@mail.ustc.edu.cn}

\author[S. Tu]{Siming Tu }
\address[Siming Tu]{ School of Mathematics (Zhuhai), Sun Yat-sen University,
	Zhuhai, Guangdong 519082, P.R. China}
\email{tusiming3@mail.sysu.edu.cn}

\author[T. Yu]{Tao Yu  }
\address[Tao Yu]{Department of Mathematics, Shantou University, Shantou 515063, P. R. China}
\email{ytnuo@mail.ustc.edu.cn}

\begin{abstract}
Using the idea of local entropy theory,  we characterize the sequence entropy tuple via  mean forms of the sensitive tuple in both topological and measure-theoretical senses.
For the measure-theoretical sense, we show that for an ergodic measure-preserving system, the
$\mu$-sequence entropy tuple, the $\mu$-mean sensitive tuple and the
$\mu$-sensitive in the mean tuple coincide, and give an example to show that the ergodicity condition is necessary.
For the topological sense, we show that for a certain class of minimal systems, the
mean sensitive tuple is the sequence entropy tuple.
\end{abstract}
\date{\today}
\maketitle
	
\section{Introduction}
By a {\it topological dynamical system} ({\it t.d.s.} for short) we mean a pair $(X,T)$, where $X$ is a compact metric space with a metric $d$ and $T$ is a homeomorphism from $X$ to itself.
A point $x\in X$ is called a  \textit{transitive point} if  ${\Orb(x,T)}=\{x,Tx,\ldots\}$ is dense in $X$.   A t.d.s. $(X,T)$ is called \textit{minimal} if all points in $X$ are transitive points.
Denote by  $\B_X$ all Borel measurable subsets of $X$. A Borel (probability) measure $\mu$ on $X$ is called $T$-\textit{invariant} if  $\mu(T^{-1}A)=\mu(A)$ for any $A\in \mathcal{B}_X$.
A $T$-invariant  measure $\mu$ on $X$ is called \textit{ergodic} if  $B\in \mathcal{B}_X$
with $T^{-1}B=B$ implies $\mu(B)=0$ or $\mu(B)=1$. Denote by $M(X, T)$ (resp. $M^e(X, T)$) the collection of all $T$-invariant measures (resp. all ergodic measures) on $X$.
For $\mu \in M(X,T)$,  the \textit{support} of $\mu$ is defined by
$\supp(\mu )=\{x\in X\colon \mu (U)>0\text{ for any neighbourhood }U\text{ of
}x\}$. Each  measure $\mu\in M(X,T)$ induces a {\it measure-preserving system} ({\it m.p.s.} for short)  $(X,\B_X,\mu, T)$.
% and $\B_m(X)$ the completion of $\B(X)$ under a probability measure $m$. Then $(X,\B_m(X),m)$ is a Lebesgue space.
%According to  Krylov-Bogolioubov theorem there exists a $T$-invariant (probability) measure $\mu$ on $\B(X)$, i.e. $\mu(T^{-1}A)=\mu(A)$ for any $A\in \mathcal{B}(X)$.

\medskip

It is well known that the entropy can be used to measure the local complexity of the structure of orbits in a given system. One may naturally ask how to characterize the entropy in a local way.
The related research started from the series of pioneering papers of Blanchard et al \cite{B1992, B1993, B1997, B1995},  in which the notions of entropy pairs and entropy pairs for a measure  were introduced.  From then on entropy pairs have been intensively studied by many researchers.  Huang and Ye \cite{HY06} extended the notions from pairs to finite tuples, and showed that
if  the entropy of a given system is positive, then there are entropy $n$-tuples
for any $n\in \N$ in both topological and measurable settings.

The sequence entropy was introduced by Ku\v shnirenko \cite{Kus} to establish the relation between spectrum theory and entropy theory. As in classical local entropy theory, the sequence entropy can also be localized.  In \cite{HLSY03, HMY04}  authors investigated the sequence entropy pairs,  sequence entropy tuples and sequence
entropy tuples for a measure, respectively.
Using tools from combinatorics, Kerr and Li \cite{KL07, KL09} studied (sequence) entropy tuples, (sequence)
entropy tuples for a measure and IT-tuples via independence sets.
Huang and Ye \cite{HY09} showed that a system has a sequence entropy $n$-tuple
if and only if its maximal pattern entropy is no less than $\log n$ in both topological and measurable settings.
More introductions and applications of the local entropy theory can refer to a survey \cite{GY09}.

\medskip
In addition to the entropy,  the sensitivity is another candidate to describe the complexity of a system, which was first
used by Ruelle \cite{Ruelle1977}. In \cite{X05}, Xiong introduced a multi-variant version of the sensitivity, called the $n$-sensitivity.
\begin{comment}
	According to  Auslander and Yorke \cite{AY80}
	a t.d.s. $(X,T)$ is called \emph{sensitive}
	if there exists $\delta>0$ such that
	for every opene (open and non-empty)
	subset $U$, there exist $x_1,x_2\in U$ and $m\in\N$ with $d(T^mx_1,T^mx_2)>\delta$.
	In \cite{X05}, Xiong introduced a multi-variate version of sensitivity, called $n$-sensitivity.
\end{comment}
Motivated by the local entropy theory, Ye and Zhang \cite{YZ08} introduced the notion of sensitive tuples. Particularly, they showed that a transitive t.d.s. is $n$-sensitive if and only if it has a sensitive $n$-tuple;
and a sequence entropy $n$-tuple of a minimal t.d.s. is a sensitive $n$-tuple.
For the converse, Maass and Shao \cite{MS07} showed that in a minimal t.d.s., if a sensitive $n$-tuple is a minimal point of  the $n$-fold product t.d.s. then it is a sequence entropy $n$-tuple.

\begin{comment}
	They introduced the notions of $n$-sensitivity for a measure $\mu$ and sensitive $n$-tuple for $\mu$ and showed that a t.d.s. with an ergodic measure $\mu$ is $n$-sensitive for $\mu$ if and only if it has a sensitive $n$-tuple for $\mu$;
	and for a t.d.s. with an ergodic
	measure $\mu$, sequence entropy $n$-tuple for $\mu$ is a sensitive $n$-tuple for $\mu$.
\end{comment}

Recently,  Li, Tu and Ye \cite{LTY15} studied the sensitivity in the mean form.
Li, Ye and Yu \cite{LY21,LYY22} further studied the multi-version of mean sensitivity and its local representation, namely, the mean $n$-sensitivity and the mean $n$-sensitive tuple. One naturally wonders if there is still a characterization of sequence entropy tuples via mean sensitive tuples.
By the results of \cite{ FGJO, GJY21,KL07,LYY22}  one can see that a sequence entropy tuple  is not always a mean sensitive tuple even in a minimal t.d.s. Nonetheless,  the works of  \cite{DG16,Huang06,LTY15} yield that every minimal mean sensitive t.d.s. (i.e. has a mean sensitive pair by \cite{LYY22}) is not tame (i.e. exists an IT pair by \cite{KL07}).
So generally, we conjecture that for any minimal t.d.s., a mean sensitive $n$-tuple  is an IT $n$-tuple and so a sequence entropy $n$-tuple by \cite[Theorem 5.9]{KL07}.
Now we can answer this question under an additional condition. Namely,\begin{thm}\label{thm:ms=>it}
	Let $(X,T)$ be a minimal t.d.s. and $\pi: (X,T)\rightarrow (X_{eq},T_{eq})$ be the factor map
	to its maximal equicontinuous factor which is almost one to one. Then for $2\le n\in\mathbb{N}$,
	$$MS_n(X,T)\subset IT_n(X,T),$$
	where $MS_n(X,T)$ denotes all  the mean sensitive $n$-tuples and $IT_n(X,T)$ denotes all the IT $n$-tuples.
\end{thm}

In the parallel measure-theoretical setting,  Huang, Lu and Ye \cite{HLY11} studied measurable sensitivity and its local representation.
The notion of $\mu$-mean sensitivity for an invariant measure $\mu$ on a t.d.s. was studied by Garc\'{\i}a-Ramos \cite{G17}.
Li \cite{L16} introduced the notion of the $\mu$-mean $n$-sensitivity,
and showed that
an ergodic m.p.s. is $\mu$-mean $n$-sensitive if and only if its maximal pattern entropy is no less than $\log n$.
The authors  in \cite{LYY22}  introduced the notion of the $\mu$-$n$-sensitivity in the mean, which was
\begin{comment}
	if there is $\delta>0$ such that for any Borel subset $A$ of $X$ with $\mu(A)>0$ there are $m\in \N$ and $n$ pairwise distinct points $x_1^m,x_2^m,\dots,x_n^m\in A$ such that
	$$
	\frac{1}{m}\sum_{k=0}^{m-1}\min_{1\le i\neq j\le n} d(T^k x_i^m, T^k x_j^m)>\delta.
	$$
	By definitions $\mu$-sensitivity in the mean tuple seems weaker than $\mu$-mean sensitivity tuple, however, they are
\end{comment}
%Recall that $(X,T)$  is \textit{$\mu$-$n$-sensitive in the mean}
proved to be equivalent  to the $\mu$-mean $n$-sensitivity in the ergodic case.
%They showed that a t.d.s. with ergodic $\mu$ is $\mu$-mean $n$-sensitive if and only if $\mu$-$n$-sensitive in the mean.

Using the idea of localization,   the authors \cite{LY21} introduced the notion of the $\mu$-mean sensitive tuple and  showed that
every $\mu$-entropy tuple of an ergodic m.p.s. is a $\mu$-mean sensitive tuple. A natural question is left open in \cite{LY21}:
\begin{ques}
	Is there a characterization of $\mu$-sequence entropy tuples  via $\mu$-mean sensitive tuples?
\end{ques}
The  authors in \cite{LT20}  introduced a weaker notion named the density-sensitive tuple and showed that every $\mu$-sequence entropy tuple  of an ergodic m.p.s. is a $\mu$-density-sensitive tuple.
In this paper, we give a positive answer to  this question. Namely,
\begin{thm}\label{cor:se=sm}
	Let $(X,T)$ be a t.d.s., $\mu\in M^e(X,T)$  and $2\le n\in \N$. Then
	the $\mu$-sequence entropy $n$-tuple, the $\mu$-mean sensitive $n$-tuple and the
	$\mu$-$n$-sensitive in the mean tuple coincide.
\end{thm}
By the definitions, it is easy to see that a $\mu$-mean sensitive $n$-tuple must be a $\mu$-$n$-sensitive in the mean tuple. Thus, Theorem \ref{cor:se=sm} is a direct corollary of the following two theorems.
\begin{thm}\label{thm:sm=>se}
	Let $(X,T)$ be a t.d.s., $\mu\in M(X,T)$ and $2\le n\in \N$. Then
	each $\mu$-$n$-sensitive in the mean tuple is a $\mu$-sequence entropy $n$-tuple.
\end{thm}

\begin{thm}\label{thm:se=>ms}
	Let $(X,T)$ be a t.d.s., $\mu\in M^e(X,T)$  and $2\le n\in \N$. Then
	each $\mu$-sequence entropy $n$-tuple is  a $\mu$-mean sensitive $n$-tuple.
\end{thm}
In fact, Theorem \ref{thm:sm=>se} shows a bit more than Theorem \ref{cor:se=sm}, as for a $T$-invariant measure $\mu$ which is not ergodic, every $\mu$-$n$-sensitive in the mean tuple is  still a $\mu$-sequence entropy $n$-tuple.
However, the following result shows that  ergodicity of $\mu$ in Theorem \ref{thm:se=>ms} is necessary.
\begin{thm}\label{thm:sm=/=se}
	For every $2\le n\in \N$, there exist a t.d.s. $(X,T)$ and $\mu\in M(X,T)$ such that there is a $\mu$-sequence entropy $n$-tuple but it is not a $\mu$-$n$-sensitive in the mean tuple.
\end{thm}

It is fair to note that  Garc{\'i}a-Ramos told us that at the same time, he with Mu{\~n}oz-L{\'o}pez had also got a  completely independent proof of the equivalence of the sequence entropy pair and the mean sensitive pair in the ergodic case \cite{GM22}. Their proof relies on the deep equivalent characterization of measurable sequence entropy pairs developed by  Kerr and Li \cite{KL09} using the combinatorial notion of  independence.  Our results provide more information in general case, and the proofs  work on the classical definition of  sequence entropy pairs introduced in \cite{HMY04}. It is worth noting that the proofs depend on a new interesting ergodic measure decomposition result (Lemma \ref{0726}), which was applied to prove the profound Erd\"os's conjectures in the number theory by Kra, Moreira, Richter and Robertson \cite{KMRR,KMRR1}. This decomposition may have more applications because it has the hybrid topological and Borel structures.

The outline of the paper is the following. In Sec. \ref{sec2}, we recall some basic notions that we
will use in the paper. In Sec. \ref{sec3}, we prove Theorem \ref{thm:sm=>se}. In Sec. \ref{sect:proof of thm se=>ms}, we show Theorem \ref{thm:se=>ms} and Theorem \ref{thm:sm=/=se}. In Sec. \ref{sec5}, we study the mean sensitive tuple and the sequence entropy in the topological sense and show Theorem \ref{thm:ms=>it}.
\section{Preliminaries}\label{sec2}
Throughout the paper, denote by $\mathbb{N}$ and ${\mathbb{Z}}_{+}$ the
collections of natural numbers  $\{1,2,\dots\}$ and non-negative integers $\{0,1,2,\dots\}$, respectively.

For $F\subset \mathbb{Z}_+$, denote by $\#\{F\}$ (or simply write $\#F$ when it is clear from the context) the cardinality  of $F$.   The \emph{upper density} $\overline{D}(F)$ of $F$  is defined by
$$
\overline{D}(F)=\limsup_{n\to\infty} \frac{\#\{F\cap[0,n-1]\}}{n}.
$$
%where $\#\{\cdot\}$ is the cardinality  of the given set.
Similarly, the \emph{lower density} $\underline{D}(F)$ of $F$ can be given by
$$
\underline{D}(F)=\liminf_{n\to\infty} \frac{\#\{F\cap[0,n-1]\}}{n}.
$$
If $\overline{D}(F)=\underline{D}(F)$,  we say that the \textit{density} of $F$ exists and equals to  the common value, which is written as $D(F)$.

Given a t.d.s. $(X,T)$ and $n\in \N$, denote by $X^{(n)}$  the $n$-fold product of $X$. Let $\Delta_n(X)=\{(x,x,\dots, x)\in X^{(n)}\colon x\in X\}$  be the diagonal of $ X^{(n)}$ and  $\Delta_n^\prime(X)=\{(x_1,x_2,...,x_n)\in X^{(n)}:  x_i=x_j  \text{ for some } 1\le i\neq j\le n \}$.

If a closed subset $Y\subset X$ is $T$-invariant in the sense of $TY= Y$,
then the restriction $(Y, T|_Y)$  (or simply  write $(Y,T)$ when it is clear from the context) is also a t.d.s., which is called a \textit{subsystem} of $(X,T)$.

Let. $(X,T)$ be a t.d.s.,  $x\in X$ and $U,V\subset X$. Denote by
$$
N(x,U)=\{n\in\Z_+ \colon T^n x\in U\} \ \text{ and }\  N(U,V)=\{n\in\Z_+: U\cap T^{-n}V\neq\emptyset\}.
$$
A t.d.s. $(X,T)$ is called \textit{transitive} if $N(U,V)\neq\emptyset$ for all non-empty open subsets $U,V$ of $X$. It is well known that the set of all transitive points in a transitive t.d.s. forms a dense $G_\delta$ subset of $X$ .

Given two t.d.s.  $(X, T)$ and $(Y,S)$,
a  map $\pi\colon X\to Y$  is called a \textit{factor map}
if $\pi$ is surjective and continuous such that  $\pi\circ T=S\circ\pi$,
and in which case $(Y,S)$ is referred to be a \textit{factor} of $(X, T)$. Furthermore,
If $\pi$ is a homeomorphism, we say  that $(X,T)$ is \textit{conjugate}  to  $(Y,S)$.

A t.d.s. $(X,T)$ is called \textit{equicontinuous} (resp. \textit{mean equicontinuous}) if for any $\ep>0$ there is $\delta>0$ such that if $x,y\in X$ with $d(x,y)<\delta$ then $\max_{k\in\Z_+}d(T^kx,T^ky)<\ep$ (resp. $\limsup_{n\to\infty}\frac{1}{n}\sum_{k=0}^{n-1}d(T^kx,T^ky)<\ep$).
Every t.d.s. $(X, T)$ is known to  have a maximal equicontinuous factor (or a maximal mean equicontinuous factor \cite{LTY15}). More studies on mean equicontinuous systems can see the recent survey \cite{LYY}.

\medskip

In the following of this section, we
fix a t.d.s. $(X,T)$ with a measure $\mu\in M(X,T)$. The {\it entropy of a finite measurable partition $\alpha=\left\{A_1, A_2, \ldots, A_k\right\}$ of $X$ }  is defined by
$
H_\mu(\alpha)=-\sum_{i=1}^k \mu\left(A_i\right) \log \mu\left(A_i\right),
$
where $0 \log 0$ is defined to be 0. Moreover, we define the {\it sequence entropy of $T$ with respect to $\alpha$ along an increasing sequence $S=\left\{s_i\right\}_{i=1}^{\infty}$ of $\mathbb{Z}_+$ } by
$$
h_\mu^{S}(T, \alpha)=\limsup _{n\rightarrow \infty} \frac{1}{n} H_\mu\left(\bigvee_{i=1}^n T^{-s_i} \alpha\right).
$$
The {\it sequence entropy of $T$ along the sequence $S$}  is
$$
h_\mu^{S}(T)=\sup _{\alpha} h_\mu^{S}(T, \alpha),
$$
where the supremum takes over all  finite measurable partitions.
Correspondingly, the {\it topological sequence entropy of $T$ with respect to $S$ and a finite open cover $\mathcal{U}$ } is
$$
h^{S}(T, \mathcal{U})=\limsup _{n \rightarrow\infty} \frac{1}{n} \log N\left(\bigvee_{i=1}^n T^{-s_i} \mathcal{U}\right),
$$
where $N\left(\bigvee_{i=1}^n T^{-s_i} \mathcal{U}\right)$ is the minimum among the cardinalities of all sub-families of
$\bigvee_{i=1}^n T^{-s_i} \mathcal{U}$ covering $X$.
The {\it topological sequence entropy of $T$ with respect to $S$ } is defined by $$h^{S}(T)=\sup _{\mathcal{U}} h^{S}(T, \mathcal{U}),$$
where the supremum takes over all  finite open covers.

Let $(x_i)_{i=1}^n\in X^{(n)}$. A finite cover  $\mathcal{U}=\{U_1,U_2,\ldots,U_k\}$ of $X$ is said to be an {\it admissible cover} with respect to $(x_i)_{i=1}^n$ if for each $1\leq j\leq k$ there exists $1\leq i_j\leq n$ such that $x_{i_j}\notin\overline{U_j}$. Analogously, we define admissible partitions with respect to $(x_i)_{i=1}^n$.

\begin{defn}[\cite{HMY04},\cite{MS07}]An $n$-tuple $(x_i)_{i=1}^n\in X^{(n)}\setminus \Delta_n(X)$, $n\geq 2$ is called
	\begin{itemize}
		\item  a sequence entropy $n$-tuple for $\mu$  if for any admissible finite Borel measurable partition $\alpha$ with respect to $(x_i)_{i=1}^n$, there exists a sequence $S=\{m_i\}_{i=1}^{\infty}$ of $\Z_+$ such that $h^{S}_{\mu}(T,\alpha)>0$. 	Denote by $SE_n^{\mu}(X,T)$ the set of all sequence entropy $n$-tuples for $\mu$.
		
		\item  a sequence entropy $n$-tuple if  for any admissible finite open cover $\mathcal{U}$ with respect to $(x_i)_{i=1}^n$, there exists a sequence $S=\{m_i\}_{i=1}^{\infty}$ of $\Z_+$ such that $h^{S}(T,\mathcal{U})>0$. 	Denote by $SE_n(X,T)$ the set of all sequence entropy $n$-tuples.
	\end{itemize}
\end{defn}	

We say that $f\in L^2(X,\B_X,\mu)$ is {\it almost periodic} if $\{f\circ T^n : n\in \mathbb{Z}_+\}$ is precompact
in $L^2(X,\B_X,\mu)$.  The set of all
almost periodic functions is denoted by $H_c$,
and there exists a $T$-invariant $\sigma$-algebra $\mathcal{K}_\mu \subset \B_X$
such that $H_c= L^2(X,\mathcal{K}_\mu,\mu)$, $\mathcal{K}_\mu$ is called the Kronecker algebra of $(X, \B_X,\mu, T )$. The product $\sigma$-algebra of $X^{(n)}$ is denoted by $\mathcal{B}_X^{(n)}$.  Define the measure $\lambda_n(\mu)$ on $\mathcal{B}_X^{(n)}$ by letting $$\lambda_n(\mu)(\prod_{i=1}^nA_i)=\int_{X}\prod_{i=1}^n\mathbb{E}(1_{A_i}|\mathcal{K}_\mu)d\mu.$$
Note that $SE_n^{\mu}(X,T)=\supp(\lambda_n(\mu))\setminus \Delta_n(X)$ \cite[Theorem 3.4]{HMY04}.

\section{Proof of Theorem \ref{thm:sm=>se}}\label{sec3}
\begin{defn}[\cite{LY21}]\label{defn:mu mean n-sensitive tuple}
For $2\le n\in \N$ and a t.d.s. $(X,T)$ with $\mu\in M(X,T)$, we say  that the $n$-tuple $(x_1,x_2,\dotsc,x_n)\in X^{(n)}\setminus \Delta_n(X)$ is
\begin{enumerate}
\item a \textit{$\mu$-mean $n$-sensitive tuple} if for any open neighbourhoods $U_i$ of $x_i$ with $i=1,2,\dotsc,n$, there is $\delta> 0$ such that for any $A\in \B_X$ with $\mu(A)>0$ there are $y_1,y_2,\dotsc,y_n\in A$ and a subset $F$ of $\Z_+$ with $\overline{D}(F)>\delta$ such that $T^k y_i \in U_i$ for all $i=1,2,\dots,n$ and $k\in F$.
\item  a \textit{$\mu$-$n$-sensitive in the mean tuple} if for any $\tau>0$, there is $\delta=\delta(\tau)> 0$ such that for any  $A\in\B_X$ with $\mu(A)>0$  there is $m\in \N$ and $y_1^m,y_2^m,\dotsc,y_n^m\in A$ such that
$$
\frac{\#\{0\le k\le m-1: T^ky_i^m\in B(x_i,\tau), i=1,2,\ldots,n\}}{m}>\delta.
$$
\end{enumerate}
\end{defn}
We denote the set of all $\mu$-mean $n$-sensitive tuples (resp.  $\mu$-$n$-sensitive in the mean tuples) by $MS_n^\mu(X,T)$ (resp. $SM_n^\mu(X,T)$). We call an $n$-tuple $(x_1,x_2,\dotsc,x_n)\in X^{(n)}$  \textit{essential} if $x_i\neq x_j$ for each $1\le i<j\le n$, and at this time we write the collection of all essential $n$-tuples in $MS_n^\mu(X,T)$ (resp.  $SM_n^\mu(X,T)$) as $MS_n^{\mu,e}(X,T)$ (resp. $SM_n^{\mu,e}(X,T)$).

\begin{comment}
\begin{defn}[\cite{LYY22}]\label{defn:mu-n-sensitive in the mean}
For $2\le n\in \N$ and a t.d.s. $(X,T)$ with $\mu\in M(X,T)$, we say that $(X,T)$  is \textit{$\mu$-$n$-sensitive in the mean}  if there is $\delta>0$ such that for any Borel subset $A$ of $X$ with $\mu(A)>0$ there are $m\in \N$ and $n$ pairwise distinct points $x_1^m,x_2^m,\dots,x_n^m\in A$ such that
$$
\frac{1}{m}\sum_{k=0}^{m-1}\min_{1\le i\neq j\le n} d(T^k x_i^m, T^k x_j^m)>\delta.
$$
\end{defn}
\end{comment}

\begin{proof}[Proof of Theorem \ref{thm:sm=>se}]
It suffices to prove $SM_n^{\mu,e}(X,T)\subset SE_n^{\mu,e}(X,T)$.
Let $(x_1,\ldots,x_n)\in SM_n^{\mu,e}(X,T)$. Take $\alpha=\{A_1,\ldots,A_l\}$ as an admissible partition of $(x_1,\ldots,x_n)$. Then for each $1\le k\le l$, there is $i_k\in \{1,\ldots,n\}$ such that $x_{i_k}\notin \overline{A_k}$. Put $E_i=\{1\le k\le l: x_i\not\in \overline{A_k}\}$ for $1\le i\le n$. Obviously, $\cup_{i=1}^n E_i=\{1,\ldots,l\}$. Set
$$B_1=\cup_{k\in E_1}A_k, B_2=\cup_{k\in E_2\setminus E_1}A_k, \ldots,   B_n=\cup_{k\in E_n\setminus(\cup_{j=1}^{n-1}E_j)}A_k.
$$
Then $\beta=\{B_1,\ldots,B_n\}$ is also an admissible partition of $(x_1,\ldots,x_n)$ such that $x_i\notin \overline{B_i}$ for all $1\le i\le n$. Without loss of generality, we assume $B_i\neq \emptyset$ for $1\le i\le n$.
It suffices to show that there exists a sequence $S=\{m_i\}_{i=1}^{\infty}$ of $\Z_+$ such that $h^{S}_{\mu}(T,\beta)>0,$ as $\alpha\succ\beta$.
Let $$h^*_\mu(T,\beta)=\sup \{h^{S}_{\mu}(T,\beta): S \ \text{is a sequence of }  \Z_+ \}.$$
By \cite[Lemma 2.2 and Theorem 2.3]{HMY04},  we have $h^*_\mu(T,\beta)=H_\mu(\beta|\mathcal{K}_\mu)$,
where $\mathcal{K}_\mu$ is the Kronecker algebra of $(X,\B_X,\mu,T)$.
So it suffices to show $\beta\nsubseteq \mathcal{K}_\mu$.

We prove it by contradiction. Now we assume that $\beta\subseteq \mathcal{K}_\mu$. Then for each $i=1,\ldots,n$, $1_{B_i}$ is an almost periodic function. By \cite[Theorems 4.7 and 5.2]{Y19}, $1_{B_i}$ is a $\mu$-equicontinuous in the mean function. That is, for each $1\le i\le n$ and  any $\tau>0$, there is a compact $K\subset X$ with $\mu(K)>1-\tau$ such that for any $\ep'>0$, there is $\delta'>0$  such that for all $m\in\N$, whenever   $x,y\in K$ with $d(x,y)<\delta'$,
\begin{equation}\label{3}
\frac{1}{m}\sum_{t=0}^{m-1}|1_{B_i}(T^tx)- 1_{B_i}(T^ty)|<\ep'.
\end{equation}

On the other hand, take $\ep>0$ such that $B_\ep(x_i)\cap B_i=\emptyset$ for  $i=1,\ldots,n$. Since $(x_1,\ldots,x_n)\in SM_n^{\mu,e}(X,T)$,  there is $\delta:=\delta(\ep)>0$ such that for any $A\in \B_X$ with $\mu(A)>0$ there are $m\in\N$ and $y_1^m,\ldots,y_n^m\in A$ such that if we denote $C_m=\{0\le t\le m-1:T^ty_i^m\in B_\ep(x_i)\text{ for all }i=1,2,\ldots,n\}$ then $\#C_m \ge m\delta$. Since $ B_\ep(x_1)\cap B_1=\emptyset$, then $ B_\ep(x_1)\subset \cup_{i=2}^nB_i$. This implies that there is $i_0\in \{2,\ldots,n\}$ such that
$$
\# \{t\in C_m: T^ty_1^m\in B_{i_0} \}\ge \frac{\#C_m}{n-1}.
$$
For any $t\in C_m$, we have $T^ty_{i_0}^m\in B_\ep(x_{i_0})$,  and then  $T^ty_{i_0}^m\notin B_{i_0}$, as  $B_\ep(x_{i_0})\cap B_{i_0}=\emptyset$. This implies that
\begin{equation}\label{e1}
	\frac{1}{m}\sum_{t=0}^{m-1}|1_{B_{i_0}}(T^ty_1^m)-1_{B_{i_0}}(T^ty_{i_0}^m)|\ge\frac{\#C_m}{m(n-1)}\ge \frac{\delta}{n-1}.
\end{equation}
Choose a measurable subset $A\subset K$ such that $\mu(A)>0$ and $\diam(A)=\sup\{d(x,y):x,y\in A\}<\delta'$, and $\ep'=\frac{\delta}{2(n-1)}$. Then by \eqref{3}, for any $m\in\mathbb{N}$ and $x,y\in A$,
$$
\frac{1}{m}\sum_{t=0}^{m-1}|1_{B_{i_0}}(T^tx)- 1_{B_{i_0}}(T^ty)|<\frac{\delta}{2(n-1)},
$$
a contradiction with \eqref{e1}.
Thus, $SM_n^{\mu,e}(X,T)\subset SE_n^{\mu,e}(X,T)$.
\end{proof}

\section{Proof of Theorem \ref{thm:se=>ms}}\label{sect:proof of thm se=>ms}
In Section 4.1, we first reduce Theorem \ref{thm:se=>ms} to just prove that it is true for the ergodic m.p.s.
 with a continuous factor map to its Kronecker factor, and  then we finish the proof of  Theorem \ref{thm:se=>ms} under this assumption. In Section 4.2, we show the condition that $\mu$ is ergodic is necessary.
\subsection{Ergodic case}
Throughout this section, we will use the following two types of factor maps between two m.p.s. $(X, \B_X,\mu, T)$ and $(Z, \B_Z,\nu, S)$.
\begin{enumerate}
\item \emph{Measurable factor maps:} a measurable map $\pi: X \rightarrow Z$ such that $\mu\circ\pi^{-1}=\nu$ and $\pi \circ T=S \circ \pi$ for $\mu$-a.e;
\item \emph{Continuous factor maps:} a topological factor map $\pi: X \rightarrow Z$ such that $\mu\circ\pi^{-1}=\nu$.
\end{enumerate}
If a continuous factor map $\pi$ such that  $\pi^{-1}(\B_Z)=\mathcal{K}_\mu$,
 $\pi$ is called a continuous factor map to its Kronecker factor.

The following result is a weaker version in \cite[Proposition 3.20]{KMRR}.
\begin{lem}\label{lem3}
Let $(X, \mathcal{B}_X,\mu, T)$ be an ergodic m.p.s. Then there exists an ergodic m.p.s. $(\tilde{X},\tilde{B}, \tilde{\mu}, \tilde{T})$ and a continuous factor map $\tilde{\pi}: \tilde{X} \rightarrow X$ such that  $(\tilde{X},\tilde{B}, \tilde{\mu}, \tilde{T})$ has a continuous factor map to its Kronecker factor.
\end{lem}

The following result shows that we only need to prove $SE_n^{\mu}(X,T)\subset MS_n^{\mu}(X,T)$ for all ergodic m.p.s. with a continuous factor map to its Kronecker factor.
\begin{lem}\label{lem5}
If $SE_n^{\tilde{\mu}}(\tilde{X},\tilde T)\subset MS_n^{\tilde{\mu}}(\tilde{X},\tilde T)$
for all ergodic m.p.s. $(\tilde{X},\tilde{B}, \tilde{\mu}, \tilde{T})$ with a continuous factor map to its Kronecker factor, then $SE_n^{\mu}(X,T)\subset MS_n^{\mu}(X,T)$ for all ergodic m.p.s. $(X, \mathcal{B}_X,\mu, T)$.
\end{lem}
\begin{proof}
By Lemma \ref{lem3}, there exists an ergodic m.p.s. $(\tilde{X},\tilde{B}, \tilde{\mu}, \tilde{T})$ and a continuous factor map $\tilde{\pi}: \tilde{X} \rightarrow X$ such that  $(\tilde{X},\tilde{B}, \tilde{\mu}, \tilde{T})$ has a continuous factor map to its Kronecker factor. Thus $SE_n^{\tilde\mu}(\tilde{X},\tilde T)\subset MS_n^{\tilde\mu}(\tilde{X},\tilde T)$, by the assumption.

For any $(x_1,\dotsc,x_n)\in SE_n^{\mu}(X,T)\setminus \Delta_n'(X)$, by \cite[Theorem 3.7]{HMY04} there exists an $n$-tuple $(\tilde{x_1},\dots,\tilde{x_n})\in SE_n^{\tilde\mu}(\tilde{X},\tilde T)\setminus \Delta_n'(\tilde{X})$ such that $\tilde\pi(\tilde{x_i})=x_i$. For any open neighborhood $U_1\times  \dots \times U_n$ of $(x_1,\dotsc,x_n)$ with
$U_i\cap U_j=\emptyset$ for $i\neq j$, then
$\tilde\pi^{-1}(U_1)\times \dots \times \tilde\pi^{-1}(U_n)$ is an open neighborhood of $(\tilde{x_1},\dots,\tilde{x_n})$.
Since $(\tilde{x_1},\dots,\tilde{x_n})\in SE_n^{\tilde\mu}(\tilde{X},\tilde T)\setminus \Delta_n'(\tilde{X})\subset MS_n^{\tilde\mu}(\tilde{X},\tilde T)\setminus \Delta_n'(\tilde{X})$,
there exists $\delta>0$ such that  for any $A\in\mathcal{B}_X$ with $\tilde{\mu}(\tilde\pi^{-1}(A))=\mu(A)>0$,
there exist $F\subset \mathbb{N}$ with $\overline{D}(F)\ge \delta$ and  $\tilde{y_1},\dots,\tilde{y_n}\in \tilde\pi^{-1}(A)$ such that for any $m\in F$,
$$
(\tilde T^m\tilde{y_1},\dots,\tilde T^m\tilde{y_n})
\in \tilde\pi^{-1}(U_1)\times \dots \times \tilde\pi^{-1}(U_n)$$
and hence $(T^m\tilde\pi(\tilde{y_1}),\dots,T^m\tilde\pi(\tilde{y_n}))\in U_1\times \dots \times U_n$. Note that $\tilde\pi(\tilde{y_i})\in A$ for each $i=1,2,\ldots,n$. Thus we have $(x_1,\dotsc,x_n)\in MS_n^{\mu}(X,T)$.
\end{proof}	
	
According to the lemma above-mentioned,	
in the rest of this section, we fix an ergodic m.p.s. with a continuous factor map $\pi:(X,\mathcal{B}_X, \mu, T)\rightarrow (Z,\mathcal{B}_Z, \nu, R)$
 to its Kronecker factor.
Moreover, we fix a measure disintegration $z \to \eta_{z}$ of $\mu$ over $\pi$,
i.e. $\mu = \int_Z \eta_{z} d\nu(z)$.

The following lemma plays a crucial role in our proof.  In \cite[Proposition 3.11]{KMRR},
the authors proved it for $n=2$, but general cases are similar in idea.  For readability, we move the complicated proof to Appendix \ref{APPENDIX}.

\begin{lem}\label{0726}
Let $\pi:(X,\mathcal{B}_X, \mu, T)\rightarrow (Z,\mathcal{B}_Z, \nu, R)$ be
 a continuous factor map to its Kronecker factor.
Then for each $n\in\N$, there exists a continuous map $\textbf{x}\mapsto \lambda_{\textbf{x}}^n$ from $X^{(n)}$ to  $M(X^{(n)})$ such that
the map $\textbf{x} \mapsto \lambda_{\textbf{x}}^n$ is an ergodic decomposition of $\mu^{(n)}$, where $\mu^{(n)}$ is the n-fold product of $\mu$ and $$	
\lambda^n_\textbf{x} = \int_Z \eta_{z + \pi(x_1)} \times\dots\times \eta_{z+\pi(x_n)} d\nu(z), \text{ for }\textbf{x}=(x_1,x_2,\ldots,x_n).$$
\end{lem}

The following two lemmas can be viewed as  generalizations of Lemma 3.3 and Theorem 3.4 in \cite{HMY04}, respectively.
\begin{lem}\label{lem1}
Let $\pi:(X,\mathcal{B}_X, \mu, T)\rightarrow (Z,\mathcal{B}_Z, \nu, R)$
be a continuous factor map to its Kronecker factor.   Assume that $\mathcal{U}=\{U_1, U_2, \dots, U_n\}$ is a measurable cover of $X$. Then  for any measurable partition $\alpha$ finer than $\mathcal{U}$ as a cover, there exists an increasing sequence $S\subset\mathbb{Z}_+$ with $h_{\mu}^{S}(T,\alpha)>0$ if and only if $\lambda_\textbf{x}^n (U_1^c\times\dots\times U_n^c)>0$ for all $\textbf{x}=(x_1,\dotsc, x_n)\in X^{(n)}$.
\end{lem}
\begin{proof}
$(\Rightarrow)$
By the contrary, we may assume that $\lambda_\textbf{x}^n(U_1^c\times\dots\times U_n^c)=0$ for some $\textbf{x}=(x_1,\dotsc, x_n)\in X^{(n)}$.
Let $C_i=\{z\in Z: \eta_{z+\pi(x_i)}(U_i^c)>0\}$ for $i=1,\dotsc,n$. Then $$\mu(U_i^c\setminus \pi^{-1}(C_i))=\int_{Z}\eta_{z+\pi(x_i)}(U_i^c\cap \pi^{-1}(C_i^c))d \nu(z)=0.$$ Put $D_i=\pi^{-1}(C_i)\cup (U_i^c\setminus \pi^{-1}(C_i))$.  Then $D_i\in \pi^{-1}(\mathcal{B}_Z)= \mathcal{K}_\mu$ and $D_i^c\subset U_i$, where $\mathcal{K}_\mu$ is the Kronecker factor of $X$.

For any $\textbf{s}=(s(1),\dotsc,s(n))\in \{0,1\}^n$, let $D_{\textbf{s}}=\cap_{i=1}^nD_i\left(s(i)\right)$, where $D_i(0)=D_i$ and $D_i(1)=D_i^c$. Set $E_1=\left(\cap_{i=1}^nD_i\right)\cap U_1 $ and $E_j=\left(\cap_{i=1}^nD_i\right)\cap( U_j\setminus \bigcup_{i=1}^{j-1}U_i)$ for $j=2,\dotsc,n$.

Consider the measurable partition
$$\alpha=\left\{D_\textbf{s}:\textbf{s}\in\{0,1\}^n\setminus\{(0,\dotsc,0)\}\right\}\cup\{E_1, \dotsc, E_n\}.$$
For any $\textbf{s}\in \{0,1\}^n\setminus\{(0,\dotsc,0)\}$, we have $s(i)=1$ for some $i=1,\dotsc,n$, then $D_\textbf{s}\subset D_i^c\subset U_i$. It is straightforward that for all $1\leq j\leq n$, $E_j\subset U_j$. Thus $\alpha$ is finer than $\mathcal{U}$ and by hypothesis there exists an increasing sequence $S$ of $\mathbb{Z}_+$ with $h_{\mu}^{S}(T,\alpha)>0$.

On the other hand, since $\lambda_{\textbf{x}}^n(U_1^c\times\dots\times U_n^c)=0$, we deduce $\nu\left(\cap_{i=1}^nC_i\right)=0$ and hence $\mu\left(\cap_{i=1}^nD_i\right)=0$. Thus we have $E_1,\dotsc, E_n\in \mathcal{K}_\mu$. It is also clear that $D_\textbf{s}\in \mathcal{K}_\mu$ for all $\textbf{s}\in\{0,1\}^n\setminus\{(0,\dotsc,0)\}$, as $D_1,\dotsc,D_n\in \mathcal{K}_\mu.$ Therefore each element of $\alpha$ is $\mathcal{K}_\mu$-measurable, by \cite[Lemma 2.2]{HMY04}, $$h^{S}_{\mu}(T,\alpha)\leq H_{\mu}(\alpha|\mathcal{K}_\mu)=0,$$ a contradiction.

$(\Leftarrow)$Assume $\lambda_\textbf{x}^n(U_1^c\times\dots\times U_n^c)>0$ for any $\textbf{x}\in X^{(n)}$. In particular, we take $\textbf{x}=(x,\ldots,x)$ such that $\pi(x)$ is the identity element of group $Z$. Without loss of generality, we may assume that any finite measurable partition $\alpha$ which is finer than $\mathcal{U}$ as a cover is of the type $\alpha=\left\{A_1, A_2, \ldots, A_n\right\}$ with $A_i \subset U_i$, for $1 \leqslant i \leqslant n$. Let $\alpha$ be one of such partitions. We observe that
\begin{equation*}
\begin{split}
\int_Z \eta_{z}({A_1^c}) \dots\eta_{z}(A_n^c) d\nu(z) \ge \int_Z \eta_{z }({U_1^c}) \dots\eta_{z}(U_n^c) d\nu(z)=\lambda_\textbf{x}^n(U_1^c\times\dots\times U_n^c)>0.
\end{split}
\end{equation*}
Therefore, $A_j \notin \mathcal{K}_\mu$ for some $1 \leqslant j \leqslant n$. It follows from \cite[Theorem 2.3]{HMY04} that there exists a sequence $S \subset \mathbb{Z}_+$ such that $h_\mu^{S}(T, \alpha)=H_\mu\left(\alpha \mid \mathcal{K}_\mu\right)>0$. This finishes the proof.
\end{proof}

\begin{lem}\label{lem2}
For any $\textbf{x}=(x_1,\dotsc,x_n)\in X^{(n)}$,
\[SE_{n}^\mu(X,T)= \operatorname{supp}\lambda_{\textbf{x}}^n\setminus \Delta_n(X).\]
\end{lem}
\begin{proof}
On the one hand, let $\textbf{y}=(y_1,\dotsc,y_n)\in SE_n^{\mu}(X,T)$. We show that $\textbf{y}\in\operatorname{supp}\lambda_{\textbf{x}}^n\setminus \Delta_n(X)$. It suffices to prove that for any measurable neighborhood $U_1\times \dots \times U_n$ of $\textbf{y}$, $$\lambda_{\textbf{x}}^n\left(U_1\times U_2\times \dots \times U_n\right)> 0.$$
Without loss of generality, we assume that $U_i\cap U_j=\emptyset$ if $y_i\not= y_j$.
Then $\mathcal{U}=\{U_1^c, U_2^c, \dots, U_n^c\}$ is a finite cover of $X$.
It is clear that any finite measurable partition $\alpha$ finer than $\mathcal{U}$ as a cover is an admissible partition with respect to $\textbf{y}$. Therefore, there exists an increasing sequence $S\subset\mathbb{Z}_+$ with $h_{\mu}^{S}(T,\alpha)>0$.
By Lemma \ref{lem1}, we obtain that $$\lambda_\textbf{x}^n\left(U_1\times U_2\times \dots \times U_n\right)> 0,$$ which implies that $\textbf{y}\in \operatorname{supp}\lambda_{\textbf{x}}^n$. Since $\textbf{y}\notin \Delta_n(X)$, $\textbf{y}\in \operatorname{supp}\lambda_{\textbf{x}}^n\setminus \Delta_n(X)$.

On the other hand, let $\textbf{y}=(y_1,\ldots,y_n) \in \operatorname{supp}\lambda_\textbf{x}^n\setminus \Delta_n(X)$. We show that for any admissible partition $\alpha=\left\{A_1, A_2, \ldots, A_k\right\}$ with respect to $\textbf{y}$ there exists an increasing sequence $S \subset \mathbb{Z}_+$ such that $h_\mu^{S}(T, \alpha)>0$.
Since $\alpha$ is an admissible partition with respect to $\textbf{y}$ there exist closed neighborhoods $U_i$ of $y_i, 1 \leqslant i \leqslant n$, such that for each $j \in\{1,2, \ldots, k\}$ we find $i_j \in\{1,2, \ldots, n\}$ with $A_j \subset U_{i_j}^c$. That is, $\alpha$ is finer than $\mathcal{U}=\left\{U_1^c, U_2^c, \ldots, U_n^c\right\}$ as a cover. Since $$\lambda_\textbf{x}^n\left(U_1\times U_2\times \dots \times U_n\right)>0,$$ by Lemma \ref{lem1}, there exists an increasing sequence $S \subset \mathbb{Z}_+$ such that $h_\mu^{S}(T, \alpha)>0$.
\end{proof}

%\begin{thm}
%$SE_n^{\mu}(X,T)\subset MS_n^{\mu}(X,T)$.
%\end{thm}

Now we are ready to give the proof of Theorem \ref{thm:se=>ms}.
\begin{proof}[Proof of Theorem \ref{thm:se=>ms}]
We only need to prove that $SE_n^{\mu,e}(X,T)\subset MS_n^{\mu,e}(X,T)$.
We let $\pi:(X,\mathcal{B}_X, \mu, T)\rightarrow (Z,\mathcal{B}_Z, \nu, R)$
be a continuous factor map to its Kronecker factor.
For any $\textbf{y}=(y_1,\ldots,y_n)\in SE_n^{\mu,e}(X,T)$, let $U_1\times U_2\times \dots \times U_n$ be an open neighborhood of $\textbf{y}$ such that $U_i\cap U_j=\emptyset$ for $1\le i\not=j \le n$.
By Lemma \ref{lem2}, one has $\lambda_\textbf{x}^n\left(U_1\times U_2\times \dots \times U_n\right)> 0$ for any $\textbf{x}=(x_1,\dotsc,x_n)\in X^{(n)}$.
Since the map  $\textbf{x} \mapsto \lambda_\textbf{x}^n$ is continuous, $X$ is compact and $U_1, U_2, \dotsc, U_n$ are open sets, it follows that there exists $\delta>0$ such that for any $\textbf{x}\in X^{(n)}$, $\lambda_\textbf{x}^n\left(U_1\times U_2\times \dots \times U_n\right)\ge \delta$. As the map $\textbf{x} \mapsto \lambda_\textbf{x}^n$ is an ergodic decomposition of $\mu^{(n)}$, there exists $B\subset X^{(n)}$ with $\mu^{(n)}(B)=1$ such that $\lambda_\textbf{x}^n$ is ergodic on $X^{(n)}$ for any $\textbf{x}\in B$.

For any $A\in\mathcal{B}_X$ with $\mu (A)>0$, there exists a subset $C$ of $X^{(n)}$ with $\mu^{(n)}(C)>0$ such that for any $\textbf{x}\in C$,
\[\lambda_\textbf{x}^n(A^n)>0.\]
Take $\textbf{x}\in B\cap C$, by the Birkhoff pointwise ergodic theorem, for $\lambda_\textbf{x}^n$-a.e. $(x_1',\dotsc,x_n')\in X^{(n)}$
\[\lim_{N\to \infty}\frac{1}{N}\sum_{m=0}^{N-1}1_{U_1\times U_2\times\dots\times U_n}(T^mx_1',\dotsc,T^mx_n')=\lambda_\textbf{x}^n\left(U_1\times U_2\times\dots\times U_n\right)\ge \delta.\]
Since $\lambda_\textbf{x}^n\left(A^n\right)>0$, there exists $(x_1'',\dotsc,x_n'')\in A^n$ such that
\begin{equation*}
\begin{split}
&\lim_{N\to \infty}\frac{1}{N}\#\{m\in[0,N-1]:(T^mx_1'',\dotsc,T^mx_n'')\in U_1\times U_2\times\dots\times U_n\}\\
&=\lim_{N\to \infty}\frac{1}{N}\sum_{m=0}^{N-1}1_{U_1\times U_2\times\dots\times U_n}(T^mx_1'',\dotsc,T^mx_n'')\\
&=\lambda_\textbf{x}^n\left(U_1\times U_2\times\dots\times U_n\right)\ge \delta.
\end{split}
\end{equation*}
Let $F=\{m\in\mathbb{Z}_+:(T^mx_1'',\dotsc,T^mx_n'')\in U_1\times U_2\times\dots\times U_n\}$. Then $D(F)\ge \delta$  and hence $\textbf{y}\in  MS_n^{\mu,e}(X,T).$ This finishes the proof.
\end{proof}

\subsection{Non-ergodic case}
\begin{lem}\label{lem4}
Let $(X,T)$ be a t.d.s. For any $\mu\in M(X,T)$ with the form $\mu=\sum_{i=1}^{m}\lambda_i\nu_i$, where $\nu_i\in M^e(X,T)$, $\sum_{i=1}^m\lambda_i=1$ and $\lambda_i>0$, one has
\begin{equation}\label{1}
\bigcup_{i=1}^mSE_n^{\nu_i}(X,T)\subset SE_n^{\mu}(X,T)
\end{equation}
and
\begin{equation}\label{2}
\bigcap_{i=1}^mSM_n^{\nu_i}(X,T)= SM_n^{\mu}(X,T).
\end{equation}
\end{lem}
\begin{proof}
We first prove that \eqref{1}. For any $\textbf{x}=(x_1,\dotsc,x_n)\in\bigcup_{i=1}^mSE_n^{\nu_i}(X,T)$, there exists $i\in\{1,2,\ldots,m\}$ such that $\textbf{x}\in SE_n^{\nu_i}(X,T)$ and then for any admissible partition $\alpha$ with respect to $\textbf{x}$, there exists $S=\{s_j\}_{j=1}^\infty$ such that
$h_{\nu_i}^S(T,\alpha)>0.$ By the definition of the sequence entropy
\[h_{\mu}^S(T,\alpha)=\limsup_{N\to \infty}\sum_{i=1}^m\lambda_i\frac{1}{N}H_{\nu_i}(\bigvee_{j=0}^{N-1}T^{-s_j}\alpha)\ge \lambda_ih_{\nu_i}^S(T,\alpha)>0.\]
So $\textbf{x}\in SE_n^{\mu}(X,T)$, which finishes the proof of \eqref{1}.

Next, we show \eqref{2}. For this, we only need to note that  for any $A\in\mathcal{B}_X$, $\mu(A)>0$ if and only if $\nu_j(A)>0$ for some $j\in\{1,2,\ldots m\}.$
\end{proof}

\begin{proof}[Proof of Theorem \ref{thm:sm=/=se}]
We first claim that there is a t.d.s. $(X,T)$ with $\mu_1,\mu_2\in M^e(X,T)$ such that $SE_n^{\mu_1}(X,T)\neq SE_n^{\mu_2}(X,T)$. For example,  we recall that the full shift on two symbols with the measure defined by the probability vector $(1/2,1/2)$. It has completely positive entropy and the measure has the full support. Thus every non-diagonal $n$-tuple is a sequence entropy $n$-tuple for this measure.
In particular, we consider  such two full shifts  $(X_1,T_1,\mu_1)=\left(\{0,1\}^{\mathbb{Z}},\sigma_1,\mu_1\right)$ and $(X_2,T_2,\mu_2)=\left(\{2,3\}^{\mathbb{Z}},\sigma_2,\mu_2\right)$
and define a new system $(X,T)$ as $X=X_1\bigsqcup X_2$, $T|_{X_i}=T_i, i=1,2$.
Then $\mu_1,\mu_2\in M^e(X,T)$ and $SE_n^{\mu_1}(X,T)=X_1^{(n)}\setminus\Delta_n(X_1)\neq X_2^{(n)}\setminus\Delta_n(X_2)=SE_n^{\mu_2}(X,T).$

Let $\mu=\frac{1}{2}\mu_1+\frac{1}{2}\mu_2\in M(X,T)$. By Lemma \ref{lem4}, if $SE_n^\mu(X,T)=SM_n^\mu(X,T)$ then we have
\[\cup_{i=1}^2SE_n^{\mu_i}(X,T)\subset SE_n^{\mu}(X,T)=SM_n^\mu(X,T)=\cap_{i=1}^2SM_n^{\mu_i}(X,T).\]
However, applying Theorem \ref{cor:se=sm} to each $\mu_i\in M^e(X,T)$, one has
\[SE_n^{\mu_i}(X,T)=SM_n^{\mu_i}(X,T), \text{ for }i=1,2.\]
So $SE_n^{\mu_1}(X,T)= SE_n^{\mu_2}(X,T)$, a contradiction with our assumption.
\end{proof}

\section{topological sequence entropy and mean sensitive tuples}\label{sec5}

This section is devoted to providing some partial evidences for the conjecture that in a  minimal system every mean sensitive tuple is a topological sequence entropy tuple.

It is known that the topological sequence entropy tuple has lift property \cite{MS07}. We can show that under the minimality condition, the mean sensitive tuple also has lift property. Let us begin with some notions.
For $2\le n\in\N$, we say that $(x_1,x_2,\dotsc,x_n)\in X^{(n)}\setminus \Delta_n(X)$ (resp. $(x_1,x_2,\dotsc,x_n)\in X^{(n)}\setminus \Delta'_n(X)$) is a  \textit{mean $n$-sensitive tuple} (resp. an \textit{essential mean $n$-sensitive tuple}) if
for any $\tau>0$,
there is $\delta=\delta(\tau)> 0$ such that for
any  nonempty open set $U\subset X$
there exist $y_1,y_2,\dotsc,y_n\in U$ such that if we denote
$F=\{k\in\Z_+\colon T^ky_i\in B(x_i,\tau),i=1,2,\ldots,n\}$ then $\overline{D}(F)>\delta$. Denote the set of all mean $n$-sensitive tuples (resp. essential mean $n$-sensitive  tuples) by $MS_n(X,T)$ (resp. $MS^e_n(X,T)$).

\begin{thm}\label{lem:MSn-factor}
	Let $\pi: (X,T)\ra (Y,S)$ be a factor map between two t.d.s. Then
	\begin{enumerate}
		\item $\pi^{(n)} ( MS_n(X,T))\subset MS_n(Y,S)\cup \Delta_n(Y)$ for every $n\geq 2$;
		\item $\pi^{(n)}\left(MS_n(X, T) \cup \Delta_n(X)\right)= MS_n(Y,S)\cup \Delta_n(Y)$  for every $n\geq 2$, provided that $(X,T)$ is minimal.
	\end{enumerate}
\end{thm}
\begin{proof}
(1) is easy to be  proved by the definition. We only prove (2).

Supposing that $(y_1,y_2,\cdots,y_n)\in MS_n(Y,S)$, we will show that there exists $(z_1,z_2,\cdots,z_n)\in MS_n(X,T)$  such that $\pi(z_i)=y_i$.
Fix $x\in X$ and let $U_m=B(x,\frac{1}{m})$. Since $(X,T)$ is minimal, $\operatorname{int}(\pi(U_m))\not= \emptyset$, where $\operatorname{int}(\pi(U_m))$ is the interior of $\pi(U_m)$.
Since $(y_1,y_2,\cdots,y_n)\in MS_n(Y,S)$, there exists $\delta>0$  and $y_m^1, \cdots, y_m^n\in \operatorname{int}(\pi(U_m))$ such that
$$\overline{D}(\{k\in \mathbb{Z}_+: S^ky_m^i \in \overline{B(y_i, 1)} \text{ for }i=1,\ldots,n\})\ge \delta.$$
Then there exist $x_m^1, \cdots, x_m^n\in U_m$ with $\pi(x_m^i)=y_m^i$ such that for any $m\in \N$,
$$\overline{D}(\{k\in \mathbb{Z}_+: T^kx_m^i \in \pi^{-1}(\overline{B(y_i, 1)})\text{ for }i=1,\ldots,n\})\ge \delta.$$

Put
$$
A=\prod_{i=1}^n \pi^{-1}(\overline{B(y_i, 1)}),
$$
and it is clear that $A$ is a compact subset of $X^{(n)}$.

We can cover $A$ with finite nonempty open sets of diameter less than $1$,
i.e., $A \subset \cup_{i=1}^{N_1}A_1^i$ and $\diam(A_1^i)<1$.
Then for each $m\in \N$ there is $1\leq N_1^m\leq N_1$ such that
$$\overline{D}(\{k\in \Z_+: (T^kx_m^1,\ldots, T^kx_m^n)\in \overline{A_1^{N_1^m}}\cap A \})\ge \delta/N_1.$$
Without loss of generality, we assume $N_1^m=1$ for all $m\in \N$. Namely,
$$
\overline{D}(\{k\in \Z_+: (T^kx_m^1,\ldots, T^kx_m^n)\in \overline{A_1^{1}}\cap A\}) \ge \delta/N_1 \text{ for all }m\in\mathbb{N}.
$$

Repeating above procedure, for $l\ge 1$ we can cover $\overline{A_l^{1}}\cap A$ with finite  nonempty open sets of diameter less than $\frac{1}{l+1}$,
i.e., $\overline{A_l^{1}}\cap A \subset \cup_{i=1}^{N_{l+1}}A_{l+1}^i$ and $\diam(A_{l+1}^i)<\frac{1}{l+1}$.
Then for each $m\in \N$ there is $1\leq N_{l+1}^m\leq N_{l+1}$ such that
$$
\overline{D}(\{k\in \Z_+: (T^kx_m^1,\ldots, T^kx_m^n)\in \overline{A_{l+1}^{N_{l+1}^m} }\cap A  \}) \ge \frac{\delta}{N_1N_2\cdots N_{l+1}}.
$$
Without loss of generality we assume $N_{l+1}^m=1$ for all $m\in \N$. Namely,
$$
\overline{D}(\{k\in \Z_+: (T^kx_m^1,\ldots, T^kx_m^n)\in \overline{A_{l+1}^{1}}\cap A  \}) \ge\frac{\delta}{N_1N_2\cdots N_{l+1}} \text{ for all }m\in\mathbb{N}.
$$

It is clear that there is a unique point $(z_1^1,\ldots,z_n^1)\in \bigcap_{l=1}^{\infty} \overline{A_l^{1}}\cap A $.
We claim that $(z_1^1,\ldots,z_n^1)\in MS_n(X, T)$. In fact,
for any $\tau>0$, there is $l\in \N$ such that $\overline{A_{l}^{1}}\cap A \subset V_{1}\times\cdots \times V_{n}$, where $V_i=B(z_i^1,\tau)$ for $i=1,\ldots,n$.
By the construction, for any $m\in\mathbb{N}$ there are $x_m^1,\ldots, x_m^n\in U_m$ such that
$$
\overline{D}(\{k\in \Z_+: (T^kx_m^1,\ldots, T^kx_m^n)\in \overline{A_{l}^{1}}\cap A  \}) \ge\frac{\delta}{N_1N_2\cdots N_{l}}
$$
and so
$$
\overline{D}(\{k\in \Z_+: (T^kx_m^1,\ldots, T^kx_m^n)\in V_{1}\times\cdots \times V_{n} \}) \ge \frac{\delta}{N_1N_2\cdots N_{l}}.
$$
for all $m\in \N$.
For any nonempty open set $U\subset X$,  since $x$ is a transitive point, there is $s\in \Z$ such that $T^sx\in U$. We can choose $m\in \Z$ such that
$T^sU_{m}\subset U$. This implies that $T^sx_{m}^1,\ldots, T^sx_{m}^n\in U$ and
$$
\overline{D}(\{k\in \Z_+: (T^k(T^sx_{m}^1),\ldots, T^k(T^sx_{m}^n))\in V_{1}\times\cdots \times V_{n}\} ) \ge \frac{\delta}{N_1N_2\cdots N_{l}}.
$$
So we have $(z_1^1,\ldots,z_n^1)\in MS_n(X, T)$.

Similarly, for each $p\in\mathbb{N}$, there exists $(z_1^p,\ldots,z_n^p)\in MS_n(X, T)\cap \prod_{i=1}^n \pi^{-1}(\overline{B(y_i, \frac{1}{p})})$.
Set $z_i^p\ra z_i$ as $p\ra \infty$. Then $(z_1,\ldots,z_n)\in MS_n(X, T)\cup \Delta_n(X)$ and  $\pi(z_i)=y_i$.
\end{proof}

Denote by $\mathcal{A}(MS_2(X, T))$  the smallest closed $T\times T$-invariant equivalence relation containing $MS_2(X, T)$.

\begin{cor}\label{cor:max-me-factor}
Let $(X,T)$ be a minimal t.d.s. Then $X/\mathcal{A}(MS_2(X, T))$ is the maximal mean equicontinuous factor of $(X,T)$.
\end{cor}

\begin{proof}
Let $Y=X/\mathcal{A}(MS_2(X, T))$ and $\pi:(X,T)\to (Y,S)$ be the corresponding factor map. We show  that $(Y,S)$ is mean equicontinuous. Assume that $(Y,S)$ is not mean equicontinuous, by \cite[Corollary 5.5]{LTY15} $(Y,S)$ is  mean sensitive.
Then by \cite[Theorem 4.4]{LYY22}, $MS_2(Y,S)\not=\emptyset$.
By Theorem \ref{lem:MSn-factor}, there exists $(x_1,x_2)\in MS_2(X, T)$ such that $(\pi(x_1),\pi(x_2))\in MS_2(Y,S)$.
Then $(x_1,x_2)\not \in R_\pi:=\{(x,x')\in X\times X:\pi(x)=\pi(x')\}$, a contradiction with $R_\pi=\mathcal{A}(MS_2(X, T))$.

Let $(Z,W)$ be a mean equicontinuous t.d.s. and $\theta: (X,T)\to (Z,W)$ be a factor map. Since $(X,T)$ is minimal, so is $(Z,W)$. Then by  \cite[Corollary 5.5]{LTY15} and \cite[Theorem 4.4]{LYY22}, $MS_2(Z,W)=\emptyset$. By  Theorem \ref{lem:MSn-factor} $MS_2(X,T)\subset R_\theta$, where $R_\theta$ is the corresponding equivalence relation with respect to $\theta$. This implies that $(Z,W)$ is a factor of $(Y,S)$ and so $(Y,S)$ is the maximal mean equicontinuous factor of $(X,T)$.
\end{proof}

\medskip
%\subsection{Proof of Theorem \ref{thm:ms=>it}}
In the following we show  Theorem \ref{thm:ms=>it}.
Let us begin with some preparations.
\begin{defn}[\cite{KL07}]Let $(X,T)$ be a t.d.s.
	\begin{itemize}
\item For a tuple $(A_1,A_2,\ldots, A_n)$ of subsets of $X$,
we say that a set $J\subseteq \Z_+$ is an {\em independence set} for $A$ if for
every nonempty finite subset $I\subseteq J$ and function
$\sigma: I\rightarrow \{1,2,\ldots, n\}$ we have $\bigcap_{k\in I} T^{-k} A_{\sigma(k)}\neq \emptyset.$

\item For $n\ge2$, we call a tuple $\textbf{x}=(x_1,\ldots,x_n)\in X^{(n)}$ an {\em IT-tuple}
if for any product neighbourhood $U_1\times U_2\times \ldots \times U_n$ of $\textbf{x}$ in $X^{(n)}$
the tuple $(U_1,U_2,\ldots, U_n)$ has an infinite independence set. We denote the set of IT-tuples of length $n$ by ${\rm IT}_n (X, T)$.

\item For $n\ge2$, we call an  IT-tuple $\textbf{x}=(x_1,\ldots,x_n)\in X^{(n)}$ an essential {\em IT-tuple}
if $x_i\neq x_j$ for any $i\neq j$. We denote the set of all essential IT-tuples of length $n$ by ${\rm IT}^e_n (X, T)$.
	\end{itemize}
\end{defn}

%To prove Theorem \ref{thm:ms=>it}, we need the following two lemmas.

\begin{prop}\cite[Proposition 3.2]{HLSY}\label{independent sets}
Let $X$ be a compact metric topological group
with the left Haar measure $\mu$, and let $n\in \N$ with $n\ge 2$. Suppose that
$V_{1},\ldots,V_{n}\subset X$ are compact subsets satisfying that
\begin{enumerate}
\item[(i)] $\overline{\operatorname{int}  V_i}=V_i$ for $i=1,2,\cdots,n$,

\item[(ii)] $\operatorname{int}(V_{i})\cap \operatorname{int}(V_{j})=\emptyset$ for all $1\le i\neq j\le n$,

\item[(iii)] $\mu(\bigcap_{1\leq i\leq n}V_{i})>0$.
\end{enumerate}
Further, assume that $T: X\rightarrow X$ is a minimal rotation and $\mathcal{G}\subset X$
is a residual set. Then there exists an infinite set $I\subset \Z_+$
such that for all $a\in\{1,2,\ldots,n\}^{I}$ there exists $x \in\mathcal{G}$
with the property that
\begin{equation}\label{eq: in the int}
x\in \bigcap_{k\in I} T^{-k} {\rm int}(V_{a(k)}),\quad {\rm i.e.}\
T^kx\in \operatorname{int}(V_{a(k)}) \ \text{ for any }k\in I.
\end{equation}
\end{prop}
A subset $Z\subset X$ is called {\it proper} if  $Z$ is  a compact subset with $\overline{\operatorname{int}(Z)} = Z$. The following lemma can help us to complete the proof of Theorem \ref{thm:ms=>it}.
\begin{lem}\label{lem:proper one to one}
Let $(X,T)$ and $(Y,S)$ be  two t.d.s., and $\pi:(X,T)\to (Y,S)$ be a factor map. Suppose that $(X,T)$ is minimal. Then the image of proper subsets of $X$ under $\pi$ are proper subset of $Y$.
\end{lem}
\begin{proof}
Given a proper subset $Z$ of $X$, we will show $\pi(Z)$ is also proper. It is clear that $\pi(Z)$ is compact, as $\pi$ is continuous.  Now we prove $\overline{\operatorname{int}(\pi(Z))} = \pi(Z)$.

It follows from the closeness of $\pi(Z)$ that $\overline{\operatorname{int}(\pi(Z))} \subset \pi(Z)$. On the other hand, for any $y\in \pi(Z), $ take $x\in \pi^{-1}(y)\cap Z$. Since  $\pi^{-1}(y)\cap Z=\pi^{-1}(y)\cap\overline{\operatorname{int}(Z)}$, there exists  a sequence $\{x_n\}_{n\in\mathbb{N}}$ such that $x_n\in \operatorname{int}(Z)$ and $\lim_{n\to \infty}x_n=x$.  Let $\{r_n\}_{n\in\mathbb{N}}$ be a sequence of $\mathbb{R}$ satisfying $$\lim_{n\to\infty}r_n=0\text{ and }B(x_n,r_n)\subset \operatorname{int}(Z).$$ By the minimality of $(X,T)$, we have $\pi$ is semi-open, and hence $\operatorname{int}(\pi(B(x_n,r_n)))\neq \emptyset$. Thus, there exists $x_n'\in B(x_n,r_n)$ such that $\pi(x_n')\in \operatorname{int}(\pi(B(x_n,r_n)))\subset\operatorname{int}(\pi(Z))$. Since $x_n'\in B(x_n,r_n)$ and $\lim_{n\to \infty}x_n=x$, one has $\lim_{n\to \infty}x_n'=x$, and hence $\lim_{n\to \infty}\pi(x_n')=\pi(x)=y.$ This implies that $y\in \overline{\operatorname{int}(\pi(Z))}$, which finishes the proof.
\end{proof}

Inspired by \cite[Proposition 3.7]{HLSY}, we can give the proof of Theorem  \ref{thm:ms=>it}.

%\begin{thm}\label{key-prop-numberofIT} Let $(X,T)$ be a minimal t.d.s. and $\pi: (X,T)\rightarrow (X_{eq},T_{eq})$ be the factor
%to the its maximal equicontinuous factor which is almost one to one. If $(x_1,x_2,\cdots,x_n)\in MS_n(X,T)$,
%then  $(x_1,x_2,\cdots,x_n)\in IT_n(X,T)$.
%
%\end{thm}
\begin{proof} [Proof  of Theorem  \ref{thm:ms=>it}]
	It suffices to prove $MS^e_n(X,T)\subset IT_n^e(X,T)$.	
Given $\textbf{x}=(x_1,\ldots,x_n)\in MS^e_n(X,T)$, we will show that $\textbf{x}\in IT^e_n(X,T).$

Since the minimal t.d.s. $(X_{eq},T_{eq})$ is the maximal equicontinuous factor of $(X,T)$,  then $X_{eq}$ can be viewed as a compact metric group with a $T_{eq}$-invariant metric $d_{eq}$. Let $\mu$ be  the left Haar probability measure of $X_{eq}$, which is also
the unique $T_{eq}$-invariant probability measure of $(X_{eq},T_{eq})$. Let
$$X_1=\{x\in X: \#\{\pi^{-1}(\pi(x))\}=1\}, \quad Y_1=\pi(X_1).$$
Then $Y_1$ is a dense $G_\delta$-set as $\pi$ is almost one to one.

Without loss of generality, assume that $\ep=\frac 14 \min_{1\le i\neq j\le n}d(x_i,x_j)$. Let $U_i=\overline{B_\ep(x_i)}$
for $1\le i\le n$. Then $U_i$  is proper for each $1\le i\le n$.
We will show that $U_1,U_2,\ldots,U_n$ is an infinite independent tuple of
$(X,T)$, i.e. there is some infinite set $I\subseteq \mathbb{Z}_+$ such that
$$\bigcap_{k\in I}T^{-k}U_{a(k)}\neq \emptyset, \  \text{for all} \  a\in \{1,2,\ldots,n\}^I.$$

Let $V_i=\pi(U_i)$ for $1\le i\le n$. By Lemma \ref{lem:proper one to one}, $V_i$ is proper
for each $i\in \{1,2,\ldots,n\}$.
We claim that ${\rm int }(V_i)\cap {\rm int}(V_j)=\emptyset$ for all $1\le i\neq j\le n$.
In fact, if there is some  $1\le i\neq j\le n$ such that
${\rm int }(V_i)\cap {\rm int}(V_j)\not=\emptyset$, then
$${\rm int }(V_i)\cap {\rm int}(V_j)\cap Y_1\not=\emptyset,$$
as $Y_1$ is a dense $G_\d$-set. Let $y\in {\rm int }(V_i)\cap {\rm int}(V_j)\cap Y_1$.
Then there are $x_i\in U_i$ and $x_j\in U_j$ such that $y=\pi(x_i)=\pi(x_j)$, which contradicts with $y\in Y_1$.

Choose a nonempty open set $W_m\subset X$ with $\operatorname{diam}(\pi(W_m))<\frac{1}{m}$ for each $m\in \N$. Since $\textbf{x}\in MS^e_n(X,T)$,
there exist $\delta>0$ and $\textbf{x}^m=(x_1^m, x_2^m,\cdots, x_n^m)\in W_m\times \dots \times W_m$ such that
$\overline{D}(N(\textbf{x}^m, U_1\times U_2\times \cdots \times U_n))\ge \delta.$
Let $\textbf{y}^m=(y_1^m,y_2^m,\cdots,y_n^m)=\pi^{(n)} (\textbf{x}^m)$. Then
$$\overline{D}(N(\textbf{y}^m, V_1\times V_2\times \cdots \times V_n))\ge \delta.$$
For $p\in \overline{D}(N(\textbf{y}^m, V_1\times V_2\times \cdots \times V_n))$,
$T_{eq}^py_i^m\in V_i$.
As $\operatorname{diam}(\pi(W_m))<\frac{1}{m}$, $d_{eq}(y_1^m,y_i^m)<\frac{1}{m}$ for $1\le i\le n$.
Note that $$d_{eq}(T_{eq}^py_1^m,T_{eq}^py_i^m)=d_{eq}(y_1^m,y_i^m)<\frac{1}{m}\quad\text{ for }1\le i\le n.$$
Let $V_i^m=B_{\frac{1}{m}}(V_i)=\{y\in X_{eq}:d_{eq}(y,V_i)<\frac{1}{m}\}$.
Then $T_{eq}^py_1^m\in \cap_{i=1}^n V_i^m$
and $$\overline{D}(N(y_1^m, \cap_{i=1}^n V_i^m))\ge \delta.$$
Since $(X_{eq},T_{eq})$ is uniquely ergodic with respect to the measure $\mu$, $\mu(\cap_{i=1}^n V_i^m)\ge \delta$.
Letting $m\to \infty$, one has $\mu(\cap_{i=1}^n V_i)\ge \delta>0.$

 By Proposition \ref{independent sets},
there is an infinite $I\subseteq \mathbb{Z}_+$ such that for all $a\in\{1,2,\ldots,n\}^{I}$ there exists $y_0\in Y_1$
with the property that
\begin{equation*}
y_0\in \bigcap_{k\in I} T_{eq}^{-k} {\rm int}(V_{a(k)}).
\end{equation*}
 Set $\pi^{-1}(y_0)=\{x_0\}$. Then
\begin{equation*}
x_0\in \bigcap_{k\in I} T^{-k} U_{a(k)},
\end{equation*}
which implies that $(x_1,x_2,\cdots,x_n)\in IT_n(X,T)$.
\end{proof}
%\begin{rem}
%In fact, we prove a little stronger result than Theorem \ref{thm:ms=>it}, as we can find an infinite subset $I$ of $\mathbb{N}$ such that $
% \bigcap_{k\in I} T^{-k} U_{a(k)}\neq \emptyset.
%$
%However, in the definition of IT-tuples, we only need that there exists an infinite subset $I$ of $\mathbb{N}$ such that for any finite subset $J$, $
%\bigcap_{k\in J} T^{-k} U_{a(k)}\neq \emptyset.
%$
%\end{rem}

\section*{Acknowledgments}

%The authors thank Jian Li and Felipe Garica-Ramos for the useful suggestions.
We thank the referee for a very careful reading and
many useful comments, which helped us to improve the paper.
Research of Jie Li is supported by NNSF of China (Grant No. 12031019); Chunlin Liu is partially supported by NNSF of China (Grant No. 12090012); Siming Tu is supported by NNSF of China (Grant No. 11801584 and No. 12171175); and Tao Yu is supported by NNSF of China (Grant No. 12001354) and STU Scientific Research Foundation for Talents (Grant No. NTF19047).

\begin{appendix}
\section{Proof of Lemma \ref{0726}}\label{APPENDIX}
In this section, we give the proof of Lemma \ref{0726}.

\begin{lem}\label{0724}
For a m.p.s. $(X,\mathcal{B}_X,\mu,T)$ with $\mathcal{K}_\mu$ its Kronecker factor, $n\in\N$
and $f_i\in L^\infty(X,\mu)$, $i=1,\dotsc,n$, we have
\[
\lim_{M \to \infty} \dfrac{1}{M} \sum_{m=1}^M \prod_{i=1}^{n} f_i( T^m x_i)
=
\lim_{M \to \infty} \dfrac{1}{M} \sum_{m=1}^M \prod_{i=1}^{n}\mathbb{E}(f_i | \mathcal{K}_\mu)(T^m x_i).
\]
\end{lem}
\begin{proof}
On the one hand, by the Birkhoff ergodic theorem, for $\textbf{x}=(x_1,\dotsc,x_n)\in X^{(n)}$,
let $F(\textbf{x})=F(x_1,\dots,x_n)=\prod_{i=1}^{n} f_i(x_i)$,
\[\lim_{M \to \infty} \dfrac{1}{M} \sum_{m=1}^M \prod_{i=1}^{n} f_i( T^m x_i)
=\lim_{M\to\infty}\dfrac{1}{M} \sum_{m=1}^{M} F\left(\left(T^{(n)}\right)^m\textbf{x}\right)=
\mathbb{E}_{\mu^{(n)}}(\prod_{i=1}^{n}f_i|I_{\mu^{(n)}})(\textbf{x}),\]
where $I_{\mu^{(n)}}=\{A\in \mathcal{B}^{(n)}_X: T^{(n)}A=A\}$.

On the other hand, following \cite[Lemma 4.4]{HMY04}, we have $(\mathcal{K}_\mu)^{\bigotimes n}=\mathcal{K}_{\mu^{(n)}}$. Then for $\textbf{x}=(x_1,\dotsc,x_n)\in X^{(n)}$,
\[\prod_{i=1}^{n}\mathbb{E}_{\mu}(f_i|\mathcal{K}_\mu)(x_i)
=\mathbb{E}_{\mu^{(n)}}(\prod_{i=1}^{n}f_i|(\mathcal{K}_\mu)^{\bigotimes n})(\textbf{x})
=\mathbb{E}_{\mu^{(n)}}(\prod_{i=1}^{n}f_i|\mathcal{K}_{\mu^{(n)}})(\textbf{x}).\]
This implies that
\begin{align*}
\lim_{M\to\infty}\dfrac{1}{M} \sum_{m=1}^{M} \prod_{i=1}^{n}  \mathbb{E}_{\mu}(f_i|\mathcal{K}_\mu)(T^mx_i)
= & \mathbb{E}_{\mu^{(n)}}(\prod_{i=1}^{n}\mathbb{E}_{\mu}(f_i|\mathcal{K}_\mu)|I_{\mu^{(n)}})(\textbf{x})\\
= &\mathbb{E}_{\mu^{(n)}}(\mathbb{E}_{\mu^{(n)}}(\prod_{i=1}^{n}f_i|\mathcal{K}_{\mu^{(n)}})|I_{\mu^{(n)}})(\textbf{x})\\
= &\mathbb{E}_{\mu^{(n)}}(\prod_{i=1}^{n}f_i|I_{\mu^{(n)}})(\textbf{x}),
\end{align*}
where the last equality follows from the fact that $I_{\mu^{(n)}}\subset\mathcal{K}_{\mu^{(n)}}.$
\end{proof}

\begin{lem}\label{0725}
Let $(Z,\B_Z,\nu,R)$ be a minimal rotation on a compact abelian group.
Then for any $n\in\mathbb{N}$ and $\phi_i\in L^\infty(Z,\nu)$,
$i=1,\dotsc,n$,,
\[\lim_{M\to\infty}\dfrac{1}{M} \sum_{m=1}^{M} \prod_{i=1}^{n} \phi_i (R^mz_i)
= \int_Z \prod_{i=1}^{n} \phi_i (z_i+z)d \nu(z) \quad\text{ for }\nu^{(n)}\text{-a.e.  }(z_1,\ldots, z_n).
\]
\end{lem}
\begin{proof}
Since $(Z,\B_Z,\nu,R)$ be a minimal rotation on a compact abelian group,
there exists $a\in Z$ such that $R^mz=z+ma$ for any $z\in Z$.

Let $F(z)=\prod_{i=1}^{n} \phi_i (z_i+z)$.
Then $F(R^me_Z)=F(ma)$ where $e_Z$ is identity element of $Z$.
Since $(Z,R)$ is minimal equicontinuous,
$(Z,\B_Z,\nu,R)$ is uniquely ergodic.
	By an approximation argument, we have, for $\nu^{(n)}$-a.e. $(z_1,\ldots, z_n)$,
\begin{align*}
\lim_{M\to\infty}\dfrac{1}{M} \sum_{m=1}^{M}\prod_{i=1}^{n} \phi_i(R^mz_i)
=&\lim_{M\to\infty}\dfrac{1}{M} \sum_{m=1}^{M}\prod_{i=1}^{n} \phi_i (z_i+ma)\\
=&\lim_{M\to\infty}\dfrac{1}{M} \sum_{m=1}^{M} F(ma)
=\lim_{M\to\infty}\dfrac{1}{M} \sum_{m=1}^{M} F(R^me_Z)\\
=&\int_Z F(z) d \nu(z)
=  \int_Z \prod_{i=1}^{n} \phi_i (z_i+z) d \nu(z).
\end{align*}
The proof is completed.
\end{proof}

\begin{proof}[Proof of Lemma \ref{0726}.]
Let  $z \mapsto \eta_z$ be  the disintegration of $\mu$ over the continuous factor map $\pi$ from $(X,\B_X,\mu,T)$ to its Kronecker factor $(Z,\B_Z,\nu,R)$. For $n\in\N$,
	define
	\begin{equation*}
	\label{eqn:lambda_2_dim_definition_for_section_2_is_this_unqiue_yet}
	\lambda^n_{\textbf{x}} = \int_Z \eta_{z + \pi(x_1)} \times\dots\times \eta_{z+\pi(x_n)} d\nu(z)
	\end{equation*}
	for every $\textbf{x}=(x_1,\dotsc,x_n) \in X^{(n)}$.
	
	We first note that for each $\textbf{x} \in X^{(n)}$ the measures $\eta_{z + \pi(x_i)}$ are defined for $\nu$-a.e.  $z \in Z$ and therefore  is well-defined.
	To prove that $\textbf{x} \mapsto \lambda^n_\textbf{x}$ is continuous first note that uniform continuity implies
	\[
	(u_1,\dotsc,u_n) \mapsto \int_Z \prod_{i=1}^{n}f_i(z + u_i) d\nu(z)
	\]
	from $Z^{(n)}$ to $\mathbb{C}$ is continuous whenever $f_i \colon  Z \to \mathbb{C}$ are continuous.
	An approximation argument then gives continuity for every $f_i \in L^\infty(Z,\nu)$.
	In particular,
	\[
	\textbf{x} \mapsto \int_Z \prod_{i=1}^{n}\mathbb{E}(f_i \mid \B_Z)(z + \pi(x_i)) d\nu(z)
	\]
	from $X^{(n)}$ to $\mathbb{C}$ is continuous whenever $f_i \in L^\infty(X,\mu)$, which in turn implies continuity of $\textbf{x} \mapsto \lambda_{\textbf{x}}^n$.
	
	To prove that $\textbf{x}\mapsto \lambda_{\textbf{x}}^n$ is an ergodic decomposition we first calculate
	\begin{equation*}
 	\int_{X^{(n)}} \int_Z \prod_{i=1}^{n}\eta_{z + \pi(x_i)}d \nu(z) d \mu^{(n)}(\textbf{x})
=\int_Z \prod_{i=1}^{n}\int_X \eta_{z + \pi(x_i)}   d \mu(x_i) d \nu(z),
	\end{equation*}
	which is equal to $\mu^{(n)}$ because all inner integrals are equal to $\mu$.
	We conclude that
	\begin{equation*}
	\label{eq_continuousergodicdecompositionofmu1}
	\mu^{(n)} = \int_{X^{(n)}}\lambda^n_\textbf{x} d \mu^{(n)}(\textbf{x}),
	\end{equation*}
	which shows $\textbf{x} \mapsto \lambda^n_\textbf{x}$ is a disintegration of $\mu^{(n)}$.
	
	We are left with verifying that
	\[
	\int_{X^{(n)}} F d \lambda^n_\textbf{x} = \mathbb{E}_{\mu^{(n)}}(F \mid I_{\mu^{(n)}})(\textbf{x})
	\]
	for $\mu^{(n)}$-a.e. $\textbf{x}\in X^{(n)}$ whenever $F \colon  X^{(n)} \to \mathbb{C}$ is measurable and bounded.
	Recall that $I_{\mu^{(n)}}$ denotes the $\sigma$-algebra of $T^{(n)}$-invariant sets.
	Fix such an $F$.
	It follows from the pointwise ergodic theorem that
	\[
	\lim_{M \to \infty} \dfrac{1}{M} \sum_{m=1}^M F( T^m x_1,\dotsc, T^m x_n)
	=
	\mathbb{E}_{\mu^{(n)}}(F \mid I_{\mu^{(n)}})(\textbf{x})
	\]
	for $\mu^{(n)}$-a.e. $\textbf{x}\in X^{(n)}$.
	We therefore wish to prove that
	\[
	\int_{X^{(n)}} F d \lambda^n_\textbf{x} = \lim_{M \to \infty} \dfrac{1}{M} \sum_{m=1}^M F( T^m x_1,\dotsc, T^m x_n)
	\]
	holds for $\mu^{(n)}$-a.e. $\textbf{x} \in X^{(n)}$.
	
	By an approximation argument it suffices to verify that
	\begin{equation*}
	\label{eqn:proving_ergodic_kk}
	\int_{X^{(n)}} f_1 \otimes\dots\otimes f_n d \lambda^n_\textbf{x}
= \lim_{M \to \infty} \dfrac{1}{M} \sum_{m=1}^M \prod_{i=1}^{n} f_i( T^m x_i)
	\end{equation*}
	holds for $\mu^{(n)}$-a.e.  $\textbf{x} \in X^{(n)}$ whenever $f_i$ belongs to $L^\infty(X,\mu)$ for $i=1,...,n$.
	
By Lemma \ref{0724},
	\[
	\lim_{M \to \infty} \dfrac{1}{M} \sum_{m=1}^M \prod_{i=1}^{n} f_i( T^m x_i)
	=
	\lim_{M \to \infty} \dfrac{1}{M} \sum_{m=1}^M \prod_{i=1}^{n}\mathbb{E}(f_i \mid \B_Z)(T^m x_i)
	\]
	for $\mu^{(n)}$-a.e.  $\textbf{x}\in X^{(n)}$.
	By Lemma \ref{0725}, for every $\phi_i$ in $L^\infty(Z,\nu)$,
	\[\lim_{M \to \infty}\dfrac{1}{M} \sum_{m=1}^{M} \prod_{i=1}^{n} \phi_i (R^mz_i)
= \int_Z \prod_{i=1}^{n} \phi_i (z_i+z)d \nu(z)\]
for $\nu^{(n)}$-a.e.  $\textbf{z}\in Z^{(n)}$.
	Taking $\phi_i = \mathbb{E}(f_i \mid \B_Z)$ gives
	\[
	\lim_{M \to \infty} \dfrac{1}{M} \sum_{m=1}^M \prod_{i=1}^{n}\mathbb{E}(f_i \mid \B_Z)(T^m x_i)
	=
	\int_{X^{(n)}} f_1 \otimes \dots \otimes f_n d \lambda^n_{\textbf{x}}
	\]
	for $\mu^{(n)}$-a.e. $\textbf{x}\in X^{(n)}$.
\end{proof}

\end{appendix}

	\end{document}